\documentclass[11pt,a4paper,reqno]{amsart}

\usepackage{amsmath}
\usepackage{amssymb}
\usepackage{amsthm}
\usepackage{enumerate}
\usepackage{amsbsy}
\usepackage{amsfonts}
\usepackage{color}
\usepackage{esint}

\usepackage[bitstream-charter]{mathdesign}
\usepackage[T1]{fontenc}

\usepackage{amsthm}
\newtheorem{thm}{Theorem}

\newtheorem{lem}{Lemma}

\theoremstyle{definition}
\newtheorem{defn}{Definition}
\theoremstyle{remark}
\newtheorem{rem}{Remark}
\newtheorem*{nota}{Notation}

\newcommand{\R}{\mathbb{R}}
\newcommand{\T}{\mathbb{T}}
\newcommand{\Z}{\mathbb{Z}}

\newcommand{\bu}{\mathbf{u}} 

\newcommand{\cP}{\mathcal{P}}

\newcommand{\tf}{\widetilde{f}}
\newcommand{\tu}{\widetilde{\mathbf{u}}}
\newcommand{\trho}{\widetilde{\rho}}

\newcommand{\p}{{\partial}}

\newcommand{\vep}{{\varepsilon}}

\newcommand{\norm}[1]{{\left\Vert#1\right\Vert}}
\newcommand{\inn}[1]{{\left\langle #1 \right\rangle}}
\newcommand{\set}[1]{{\left\lbrace #1 \right\rbrace}}

\newcommand{\intO}{\int_{\Omega}}

\DeclareMathOperator*{\divg}{div}

\begin{document}
\baselineskip=18pt

\title[Blow up criteria]{Blow up criteria for the compressible Navier--Stokes equations}

\author{Hi Jun Choe \& Minsuk Yang}

\address{Hi Jun Choe: Department of Mathematics, Yonsei University, 50 Yonseiro, Seodaemungu, Seoul, Republic of Korea}
\email{choe@yonsei.ac.kr}

\address{Minsuk Yang: School of Mathematics, Korea Institute for Advanced Study, 85 Hoegiro Dongdaemungu, Seoul, Republic of Korea}
\email{yangm@kias.re.kr}

\begin{abstract}
We study the strong solution to the 3-D compressible Navier--Stokes equations.
We propose a new blow up criterion for barotropic gases in terms of the integral norm of density $\rho$ and the divergence of the velocity $\bu$ without any restriction on the physical viscosity constants.
Our blow up criteria can be seen as a partial realization of the underlying principle that the higher integrability implies the boundedness and then eventual regularity. 
We also present similar blow up criterion for the heat conducting gases.

\end{abstract}

\maketitle

\section{Introduction}
\label{S1}

The motion of the compressible barotropic gases is governed by the viscous compressible Navier--Stokes equations, which consist of the equation of the conservation of mass 
\begin{equation}
\label{E11}
\p_t \rho + \divg(\rho \bu) = 0
\end{equation}
and the equation of the conservation of momentum 
\begin{equation}
\label{E12}
\p_t (\rho \bu) + \divg(\rho \bu \otimes \bu) = L\bu - \nabla p 
\end{equation}
in $(0,T)\times\Omega$ where the state variables $\bu$ and $\rho$ denotes the velocity field and the density of the fluid.
The Lam\'e operator $L$ is given by 
\begin{equation}
\label{E13}
L\bu = \mu \Delta \bu + (\lambda+\mu) \nabla \divg \bu
\end{equation}
where the viscosity constants satisfy the physical conditions
\begin{equation}
\label{E14}
\mu >0 \quad \text{ and } \quad 3\lambda + 2 \mu >0.
\end{equation}
The pressure $p$ is given by the barotropic constitutive relation
\begin{equation}
\label{E15}
p = \alpha \rho^\gamma 
\end{equation}
for some constants $\gamma > 1$ and $\alpha > 0$.
In this paper, the space domain is assumed to be a torus $\Omega = \T^3 = \R^3/\Z^3$.
The periodic boundary conditions are assumed and the initial conditions 
\[\rho|_{t=0} = \rho_0, \qquad \bu|_{t=0} = \bu_0\]
are assumed to guarantee the existence of strong solutions.
The constant $\alpha$ in \eqref{E15} plays no role in our argument, so we simply assume $\alpha=1$.

We also consider the motion of the compressible heat conducting gases. 
The viscous compressible heat conducting Navier--Stokes equations satisfy
\begin{equation}
\label{E16}
\p_t (\rho \theta) + \divg(\rho\theta \bu)+\rho\theta\divg \bu
-\kappa \Delta \theta = \frac{\mu}{2} |\nabla \bu +\nabla \bu^{t}|^2 +
\lambda |\divg \bu|^2 
\end{equation}
with the conservation of mass \eqref{E12} and the conservation of momentum \eqref{E13} in $(0,T)\times\Omega$ where the constant $\kappa>0$ denotes the heat conductivity.
The pressure $p$ is given by the constitutive relation
\begin{equation}
\label{E17}
p = \alpha\rho\theta.
\end{equation}
for some constant $\alpha > 0$.
The space domain is also assumed to be a torus $\Omega = \T^3 = \R^3/\Z^3$.
The periodic boundary conditions are assumed and the initial conditions 
\[\rho|_{t=0} = \rho_0, \qquad \bu|_{t=0} = \bu_0, \qquad \theta|_{t=0} = \theta_0\]
are assumed to guarantee the existence of strong solutions.
Since the constant $\alpha$ and $\kappa$ plays no role in our argument, we simply assume $\alpha=\kappa=1$.

When the initial data are in suitable energy classes, global existence of weak solutions for the compressible barotropic gases was proved for $\gamma \ge 9/5$ by Lions \cite{L}.
Later, the result was extended to the case $\gamma > 3/2$ by Feireisl--Novotn\'y--Petzlotv\'a \cite{FNP}.
Although there are huge literature on existence results according to various conditions on the initial data (e.g., Solonnikov \cite{S}, Matsumura--Nishida \cite{MN}, Vaigant--Kazhikhov \cite{VK}, and Hoff \cite{H}), we only mention the works closely related to this paper.

In a series of papers \cite{CCK,CK1,CK2,CJ} Cho, Choe, Jin, and Kim proved short time existence of strong solutions allowing vacuum region.
When vacuum region initially exists, Xin \cite{X} proved that the smooth solution blows up in finite time.
Huang--Li--Xin \cite{HLX} proved Beal-Kato-Majda type blow up criterion in terms of velocity.
Li--Li--Xin \cite{LLX} proved the strong solution blows up as vacuum states vanish.
Recently, Sun--Wang--Zhang \cite{SWZ} presented a nice characterization of the maximal existence time of strong solution for barotropic gases in terms of the essential superemum norm of density under the condition that the viscosity constants $\lambda$ and $\mu$ satisfy $\lambda < 7\mu$.
Wen--Zhu \cite{WZ} improved and extended the result in Sun--Wang--Zhang \cite{SWZ} by making some weaker assumptions on such viscosity constants.
We refer the reader to the paper \cite{WZ} for summary and references of other important blow up criteria and extensions.

Our aim of this paper is to present new blow up criteria, Theorem \ref{G1} and Theorem \ref{G2}, without any restriction on the physical constants $\lambda$ and $\mu$. 
We use the integral norm of density $\rho$ and the divergence of the velocity $\bu$ which is rather significant if we consider the continuity equation and the weak compactness for approximations.
It seems that the condition on the divergence of the velocity $\bu$ commensurate with such viscosity restrictions.
We also present similar blow up criteria for the heat conducting gases.
Our blow up criteria can be seen as a partial realization of the underlying principle that the higher integrability implies the boundedness and then eventual regularity.

\section{Preliminaries and main results}
\label{S2}

We shall use the standard notation $L^q(\Omega)$ and $W^{k,q}(\Omega)$ for the Lebesgue spaces and the Sobolev spaces.
We denote $H^{k}(\Omega) = W^{k,2}(\Omega)$ and 
\[D^{k,q}(\Omega) = \set{\bu \in L^1_{loc}(\Omega) : \norm{\nabla^k \bu}_{L^q(\Omega)} <\infty}.\]
We define the space $D_0^{k,q}(\Omega)$ to be the closure of $C_c^\infty(\Omega)$ in $D^{k,q}(\Omega)$.

The following results on the local wellposedness in time and blow up criteria of the strong solution for the initial data with vacuum are established in a series of papers \cite{CCK}, \cite{CK1}, and \cite{CK2} by Cho, Choe, and Kim.

We recall the following theorem for the existence of the strong solution corresponding to the barotropic gases.

\begin{thm}
\label{T1}
Let $\Omega$ be a bounded smooth domain or $\R^3$.
Suppose $\bu_0 \in D_0^1(\Omega)\cap D^2(\Omega)$ and 
\[\rho_0 \in W^{1,q}(\Omega) \cap H^1(\Omega)\cap L^1(\Omega)\]
for some $q\in (3,6]$.
If $\rho_0$ is nonnegative and the initial data satisfy the compatibility condition
\[L \bu_0 +\nabla p(\rho_0) = \sqrt{\rho_0} g\]
for some vector field $g\in L^2(\Omega)$, then there exist a time $T \in (0,\infty]$ and unique solution
\[(\rho,\bu)\in C([0,T] : H^1\cap W^{1,q}(\Omega)) \times C([0,T] : D^2(\omega)) \cap L^2(0,T;D^{2,q}(\Omega)).\]
Moreover, if the maximal existence time $T^\star$ of the solution is finite, then 
\[\lim\sup_{t\rightarrow T^\star} ||\rho||_{W^{1,q}}(t)+||\bu||_{D^1(\Omega)}=\infty.\]
\end{thm}

We recall the following theorem for the existence of the strong solution corresponding to the heat conducting gases.

\begin{thm}
\label{T2}
Let $\Omega$ be a bounded smooth domain or $\R^3$.
Suppose $\bu_0, \theta_0 \in D_0^1(\Omega)\cap D^2(\Omega)$ and 
\[\rho_0 \in W^{1,q}(\Omega) \cap H^1(\Omega)\cap L^1(\Omega)\] 
for some $q\in (3,6]$. 
If $\rho_0$ is nonnegative and the initial data satisfy the compatibility condition
\begin{align*}
&L \bu_0 +\nabla p(\rho_0) = \sqrt{\rho_0} g_1\\
&\Delta\theta_0+\frac{\mu}{2}|\nabla \bu_0 +(\nabla \bu_0)^{t}|^2
+ \lambda(\divg \bu_0)^2=\sqrt{\rho_0} g_2
\end{align*}
for a vector field $g_1,g_2\in L^2(\Omega)$. 
Then there exist a time $T\in (0,\infty]$ and unique solution
\begin{align*}
\label{Hstrong}
(\rho,\bu,\theta) &\in C([0,T ):H^1\cap W^{1,q}) \times C([0,T):D^2\cap D^1_0)\times L^2([0,T):D^{2,q}) \\
(\rho_t,\bu_t,\theta_t) &\in C([0,T ):L^2\cap L^q)\times L^2([0,T): D^1_0)\times L^2([0,T):D^1_0) \\
(\rho^{1/2}\bu_t,\rho^{1/2}\theta_t) &\in L^\infty([0,T):L^2) \times L^\infty([0,T):L^2).
\end{align*}
\end{thm}

Throughout the paper we consider the strong solutions in the above regularity function classes.
The result in Theorem \ref{T1} and \ref{T2} continue to hold for $\Omega = \T^3$ with periodic boundary conditions.
Our blow up criterion for the compressible barotropic gases is the following theorem.

\begin{thm}
\label{G1}
Suppose $(\rho,\bu)$ is the unique strong solution in Theorem \ref{T1}.
If the maximal existence time $T^\star$ is finite, then there exists a number $\beta \in (1,\infty)$ depending only on $\gamma$ such that 
\[\limsup_{t\rightarrow T^\star} \left(\norm{\rho}_{L^\beta}(t) + \norm{\divg \bu}_{L^3}(t)\right) =\infty.\]
\end{thm}

\begin{rem}
For example, if $\gamma=3/2$, then the integral norm of density with the integrable exponent $\beta \ge 16$ blows up.
Actually, we propose a range of the exponent $\beta$ in the proof.
Obviously, it is an important question to know the minimal integrability exponent $\beta$.
\end{rem}

Our blow up criterion for the compressible heat conducting gases is the following theorem.

\begin{thm}
\label{G2}
Suppose $(\rho,\bu,\theta)$ is the unique strong solution in Theorem \ref{T2}.
If the maximal existence time $T^*$ is finite, then there exists a number $\delta \in (1,\infty)$ depending only on fluid dynamic parameters such that 
\[\limsup_{t\rightarrow T^*}
\left(\norm{\rho}_{L^\delta}(t) + \norm{\Delta \theta}_{L^{2}}(t) + \norm{\divg \bu}_{L^3}(t)\right) = \infty.\]
\end{thm}

We end this section by giving a few notations and definitions.

\begin{defn}
We denote by $\tf$ the derivative 
\[\tf = (\p_t + u\cdot\nabla)f.\]
In many paper the notation $\dot{f}$ was used, but we use $\tf$ for visibly easy discrimination.
\end{defn}

\begin{nota}
\begin{itemize}
\item
We denote the $L^q(\Omega)$ norm simply by 
\[\norm{f}_q = \norm{f}_{L^q(\Omega)}.\]
\item
We denote $X \lesssim Y$ if there is a generic positive constant $C$ such that $|X| \le C|Y|$.
\end{itemize}
\end{nota}

\section{Reduction of the problem}
\label{S3}

If we have proved $\sup_{0 \le t \le T} \norm{\rho}_\infty < \infty$ without assuming $7\mu >\lambda$, then the argument in Section 5 of \cite{SWZ} immediately yields 
\[\rho \in L^\infty(0,T;W^{1,q}), \qquad q \in (3,6].\]
This implies that $T$ should not be the maximal existence time of the strong solution.
So, the goal of this section is to realize that the integrability of $\rho$ and $\nabla \bu$ yields the boundedness of $\rho$.

It is well known that the effective viscous flux plays an important role in the existence theory of weak solution. 
The following simple observation is very useful to our argument.
Let $\cP $ be the Helmholz--Leray projection to solenoidal spaces.
Then taking the operator $\cP$ to the momentum equation $\rho \tu = L\bu - \nabla p$ yields 
\begin{equation}
\label{E21}
\cP (\rho \tu ) = \mu \Delta \cP \bu.
\end{equation}
Let us define 
\[G := (\lambda+2\mu) \divg \bu - p.\]
Then 
\[\Delta G = \divg(\rho \tu) = \divg (L\bu-\nabla p) = \Delta \left((\lambda+2\mu) \divg \bu - p\right).\]
Then observing that $\lambda+2\mu>0$ and 
\[\divg L\bu = (\lambda+2\mu) \Delta \divg \bu,\]
we may take the divergence operator to the momentum equation to obtain 
\begin{equation}
\label{E22}
\divg(\rho \tu) = \Delta G
\end{equation}
in the sense of distributions.

The following lemma shows a sufficient condition.

\begin{lem}
\label{T31}
Suppose 
\begin{equation}
\label{E31}
\sup_{0 \le t \le T} \left(\norm{\nabla \bu}_A + \norm{\rho}_\beta + \norm{\rho^{1/C}\bu}_C\right) < \infty
\end{equation}
for some positive numbers $A,B,C,\beta$ satisfying $A < B < C$, $1/A + 1/B < 1/3$, and 
\begin{equation}
(C-1)B/(C-B) \le \beta.
\end{equation}
Then 
\[\sup_{0 \le t \le T} \norm{\rho}_\infty < \infty.\] 
\end{lem}

\begin{proof}
We define the Lagrangean flow $X$ of $\bu$ by $X(s,s,x) = x$ and 
\[\frac{\partial}{\partial t} X(t,s,x) = \bu(t, X(t,s,x)).\]
From the continuity equation we have 
\[\trho = (\p_t +\bu\cdot\nabla) \rho = -\rho \divg \bu.\]
Hence we can write 
\begin{align*}
\rho(t,X(t,0,x)) 
&= \rho_0(x) \exp\left(\int_0^t \frac{d}{ds} \ln \rho(s,X(s,0,x)) ds\right) \\
&= \rho_0(x) \exp\left(- \int_0^t (\divg \bu)(s,X(s,0,x)) ds\right). 
\end{align*}
Since $(\lambda+2\mu) \divg \bu = G + p$ and the pressure $p$ is non-negative, we can deduce 
\[\rho(t,X(t,0,x)) 
\le \rho_0(x) \exp\left(- (\lambda+2\mu)^{-1} \int_0^t G(s,X(s,0,x)) ds\right).\] 
To conclude the lemma, it suffices to prove that the following integral is bounded
\[K(x) := - \int_0^t G(s,X(s,0,x)) ds.\]

A direct computation shows 
\[\rho\tu = \widetilde{\rho \bu} - \trho \bu = \widetilde{\rho \bu} + \rho \divg \bu \bu.\]
From the relation \eqref{E22} we have
\begin{align*}
K(x)
&= - \Delta^{-1} \divg \int_0^t \widetilde{\rho \bu}(s,X(s,0,x)) ds \\
&\quad - \int_0^t \Delta^{-1} \divg (\rho \divg \bu \bu)(s,X(s,0,x)) ds
\end{align*}
where we notice that $\Delta^{-1} \divg (\widetilde{\rho \bu})$ and $\Delta^{-1} \divg (\rho \divg \bu \bu)$ are well-defined.
Integrating $\widetilde{\rho \bu}$ on the flow can be calculated as 
\begin{align*}
\int_0^t \widetilde{\rho \bu}(s,X(s,0,x)) ds 
&= \int_0^t \frac{d}{ds} (\rho \bu)(s,X(s,0,x)) ds \\
&=  (\rho \bu)(t,X(t,0,x)) - (\rho \bu)(0,x).
\end{align*}
Thus, we have 
\begin{align*}
K(x) 
&\le \Delta^{-1} \divg(\rho_0 \bu_0)(x) - \Delta^{-1} \divg(\rho \bu)(t,X(t,0,x)) \\
&\quad - \int_0^t \Delta^{-1} \divg (\rho \divg \bu \bu)(s,X(s,0,x)) ds.
\end{align*}
Since $3 < AB/(A+B)$, we have by the Sobolev-type estimate
\[\norm{\Delta^{-1} \divg (\rho \divg \bu \bu)}_\infty 
\lesssim \norm{\rho \divg \bu \bu}_{AB/(A-B)} 
\le \norm{\nabla \bu}_A \norm{\rho \bu}_B.\]
Thus, using the condition \eqref{E31}, we have 
\begin{align*}
\norm{K}_\infty 
&\lesssim \sup_{0 \le t \le T} \norm{\Delta^{-1} \divg(\rho \bu)}_\infty + \int_0^T \norm{\Delta^{-1} \divg\left(\rho \divg \bu \bu\right)}_\infty \\
&\lesssim \left(1 + T \sup_{0 \le t \le T} \norm{\nabla \bu}_A\right) \sup_{0 \le t \le T} \norm{\rho \bu}_B \\
&\lesssim \sup_{0 \le t \le T} \norm{\rho \bu}_B
\end{align*}
and for all $0 \le t \le T$
\begin{align*}
\norm{\rho \bu}_B
&\le \norm{\rho^{(C-1)/C}}_{CB/(C-B)} \norm{\rho^{1/C}\bu}_C \\
&= \norm{\rho}_{(C-1)B/(C-B)}^{(C-1)/C} \norm{\rho^{1/C}\bu}_C < \infty.
\end{align*}
This proves the lemma.
\end{proof}

\section{Proof of Theorem \ref{G1}}
\label{S4}

We assume that $(\rho, \bu)$ is the strong solution with the regularity stated in Theorem \ref{T1} and satisfies the following integrability condition for some positive number $\beta$ 
\begin{equation}
\label{E40}
\sup_{0 \le t \le T} \left(\norm{\rho}_{\beta} + \norm{\divg \bu}_3\right) < \infty.
\end{equation}
The number $\beta$ will be specified in each lemmas.

If $(\rho,\bu)$ is a strong solution in $[0,T)$ with the regularity stated in Theorem \ref{T1}, then it is easy to see from the continuity equation that for all $t \in [0,T)$
\[\intO \rho(t) = \intO \rho_0.\]
On the other hand, multiplying the moment equation by $\bu$, we see that for all $t \in [0,T)$
\begin{align*}
&\frac{1}{2} \intO \rho |\bu|^2(t) + \frac{1}{\gamma-1} \intO \rho^\gamma(t) + \mu \int_0^t \intO |\nabla \bu|^2 + (\mu+\lambda) \int_0^t \intO |\divg \bu|^2 \\
&= \frac{1}{2} \intO \rho_0 |\bu_0|^2 + \frac{1}{\gamma-1} \intO \rho_0^\gamma.
\end{align*}
We shall show that the condition \eqref{E40} with a suitable $\beta$ yields $\sup_{0 \le t \le T} \norm{\rho^{1/C}\bu}_C < \infty$.
In order to prove it, we derive the following estimate by testing with $C |\bu|^{C-2} \bu$ in the momentum equation.

\begin{lem}
\label{T41}
Let $2 < C < \infty$.
For all $t \in [0,T]$
\[\frac{d}{dt} \intO \rho |\bu|^C 
+ C \mu \intO |\bu|^{C-2} |\nabla \bu|^2 
\lesssim \left(\intO |\divg \bu|^3\right)^{C/3} + \intO p^2 |\bu|^{C-2}.\]
\end{lem}

\begin{proof}
A direct calculation gives
\[\intO \rho \p_t |\bu|^C = C \intO \rho \p_t \bu \cdot |\bu|^{C-2} \bu.\]
Using the continuity equation and integrating by parts we have 
\[\intO \p_t \rho |\bu|^C 
= - \intO \divg (\rho \bu) |\bu|^C 
= C \intO \rho (\bu \cdot \nabla) \bu \cdot |\bu|^{C-2} \bu.\]
Summing the two identity yields 
\[\frac{d}{dt} \intO \rho |\bu|^C
= C \intO \rho \tu \cdot |\bu|^{C-2} \bu.\]
Integrating by parts yields 
\begin{align*}
- \intO L\bu \cdot |\bu|^{C-2} \bu 
&= \mu \intO |\bu|^{C-2} |\nabla \bu|^2 
+ (C-2) \mu \intO |\bu|^{C-2} |\nabla |\bu||^2 \\
&\quad + (\lambda+\mu) \intO |\bu|^{C-2} |\divg \bu|^2 
+ (\lambda+\mu) \intO (\divg \bu) \bu \cdot \nabla |\bu|^{C-2}.
\end{align*}
Hence, we have 
\begin{equation}
\label{E41}
\begin{split}
&\frac{1}{C} \frac{d}{dt} \intO \rho |\bu|^C
+ \mu \intO |\bu|^{C-2} |\nabla \bu|^2 \\
&\quad + (C-2)\mu \intO |\bu|^{C-2} |\nabla |\bu||^2 
+ (\lambda +\mu) \intO |\bu|^{C-2} |\divg \bu|^2 \\
&= - (\lambda+\mu) \intO (\divg \bu) \bu \cdot \nabla |\bu|^{C-2} 
- \intO \nabla p \cdot |\bu|^{C-2}\bu \\
&=: I_1 + I_2.
\end{split}
\end{equation}
By the Sobolev inequality 
\begin{align*}
I_1 
&\le (C-2) (\lambda+\mu) \intO |\divg \bu| |\bu|^{C-2} |\nabla |\bu|| \\
&\lesssim (C-2) \left(\intO |\divg \bu|^3\right)^{1/3} 
\left(\intO |\bu|^{3C}\right)^{(C-2)/6C} 
\left(\intO |\bu|^{C-2} |\nabla |\bu||^2\right)^{1/2} \\
&\lesssim (C-2) C^{(C-2)/C} \left(\intO |\divg \bu|^3\right)^{1/3} 
\left(\intO |\bu|^{C-2} |\nabla |\bu||^2\right)^{(C-1)/C}.
\end{align*}
By the Young inequality we have for $\vep>0$ 
\[I_1 \le \vep \intO |\bu|^{C-2} |\nabla |\bu||^2 + \alpha_{\vep} (C-2)^C C^{(C-2)} \left(\intO |\divg \bu|^3\right)^{C/3}\]
and 
\begin{align*}
I_2 
&= \intO p |\bu|^{C-2} \divg \bu + p \bu \cdot \nabla |\bu|^{C-2} \\
&\le \vep \intO (|\bu|^{C-2} |\divg \bu|^2+|\bu|^{C-2} |\nabla|\bu||^2) + \alpha_{\vep} \intO p^2 |\bu|^{C-2}.
\end{align*}
for some constant $\alpha_{\vep}$.
Combining the identity \eqref{E41} and the estimates for $I_1$ and $I_2$ with a small fixed number $\vep$, we conclude the result.
We note that all the integrals in the proof converge.
\end{proof}

\begin{lem}
\label{T42}
Let $2 < C < \infty$.
If the condition \eqref{E40} holds for some $\beta$ satisfying 
\[\frac{2C\gamma-C+2}{2} \le \beta,\] 
then 
\[\sup_{0 \le t \le T} \norm{\rho^{1/C}\bu}_C < \infty.\]
\end{lem}

\begin{proof}
By Lemma \ref{T41} we have for all $t \in [0,T]$
\[\frac{d}{dt} \intO \rho |\bu|^C \lesssim 1 + \intO p^2 |\bu|^{C-2}.\]
Since $p=\rho^\gamma$, we have by the H\"older inequality 
\[\intO p^2 |\bu|^{C-2} 
\le \left(\intO \rho^{(2C\gamma-C+2)/2}\right)^{2/C} 
\left(\intO \rho |\bu|^C\right)^{(C-2)/C}.\]
Young's inequality yields
\[\frac{d}{dt} \intO \rho |\bu|^C \lesssim \intO \rho |\bu|^C + 1\]
The result follows from the Gronwall lemma.
\end{proof}

\begin{lem}
\label{T43}
Let $8 \le C < \infty$.
If the condition \eqref{E40} holds for some $\beta$ satisfying 
\[\max\set{\frac{2C\gamma-C+2}{2}, \frac{3C-6}{C-6}} \le \beta,\]
then 
\[\int_0^T \intO \rho |\tu|^2 + \sup_{0 \le t \le T} \intO |\nabla \bu|^2< \infty.\]
\end{lem}

\begin{proof}
We divide the proof into a few steps.
\begin{enumerate}[\bf{Step} 1)]
\item
Testing with $\p_t \bu$ in the equation \eqref{E12} and integrating by parts yields
\begin{align*}
0
&= \intO (\rho \tu - L\bu + \nabla p) \cdot \p_t \bu \\
&= \intO \rho \tu \cdot (\tu - (\bu\cdot\nabla)\bu) 
- \intO L\bu \cdot \p_t \bu  - \intO p \p_t \divg \bu \\
&= \intO \rho |\tu|^2 
+ \frac{1}{2} \frac{d}{dt} \intO \left(\mu |\nabla \bu|^2 +(\lambda+\mu) |\divg \bu|^2 - 2p \divg \bu\right) \\
&\quad + \intO \p_t p  \divg \bu 
- \intO \rho \tu \cdot  (\bu\cdot\nabla)\bu.
\end{align*}
Since $(\lambda+2\mu) \divg \bu = p+G$ and $\lambda+2\mu>0$, we have 
\begin{align*}
&(\lambda+2\mu) \intO \p_t p \divg \bu \\
&= \frac{1}{2} \frac{d}{dt} \intO p^2 + \intO \p_t p G \\
&= \frac{1}{2} \frac{d}{dt} \intO p^2 
+ \intO (\p_t p + \divg (p\bu)) G 
+ \intO p\bu \cdot \nabla G. 
\end{align*}	
Thus, 
\begin{equation}
\label{E42}
\begin{split}
&\intO \rho |\tu|^2 + \frac{1}{2} \frac{d}{dt} \intO \left(\mu |\nabla \bu|^2 +(\lambda+\mu) |\divg \bu|^2 - 2p\divg \bu + \frac{1}{\lambda+2\mu} p^2\right) \\
&= - \frac{1}{\lambda+2\mu} \intO (\p_t p + \divg (p\bu)) G 
- \frac{1}{\lambda+2\mu} \intO p\bu \cdot \nabla G 
+ \intO \rho \tu \cdot (\bu \cdot \nabla) \bu \\
&=: I_1 + I_2 + I_3.
\end{split}
\end{equation}
\item
Since $p = \rho^\gamma$ and $\p_t \rho = - \divg (\rho \bu)$, we have 
\[\p_t p + \divg (p\bu) = - (\gamma-1) p \divg \bu\]
and hence 
\[I_1 \lesssim \intO |p \divg \bu G|.\]
Because of the periodic boundary condition, we have $\intO \divg \bu  = 0$ and hence by the Jensen inequality $\inn{G} = - \inn{p} \lesssim \norm{p}_{3/2}$. 
Using the Sobolev-type inequality and the relation $\divg(\rho \tu) = \Delta G$ we obtain that 
\begin{align*}
I_1 
&\le \intO |p\divg \bu||G-\inn{G}| + |\inn{G}| \intO |p\divg \bu| \\
&\lesssim \norm{p}_{(3C-6)/C} 
\norm{\divg \bu}_3 
\norm{\nabla G}_{(3C-6)/(2C-6)} 
+ \norm{p}_{3/2}^2 \norm{\divg \bu}_3 \\
&\lesssim \norm{p}_{(3C-6)/C} \norm{\rho \tu}_{(3C-6)/(2C-6)}
+ \norm{p}_{3/2}^2.
\end{align*}
An elementary calculation shows that for $8 \le C < \infty$
\[\frac{\gamma (3C-6)}{C} \le \frac{2C\gamma-C+2}{2} \le \beta\] 
and hence $\norm{p}_{3/2} \lesssim \norm{p}_{(3C-6)/C} < \infty$.
H\"older's inequality yields
\begin{align*}
\norm{\rho \tu}_{(3C-6)/(2C-6)} 
&\le \norm{\rho^{1/2}}_{(6C-12)/(C-6)} \norm{\rho^{1/2} \tu}_2 \\
&= \norm{\rho}_{(3C-6)/(C-6)}^{1/2} \norm{\rho^{1/2} \tu}_2 \\
&\lesssim \norm{\rho^{1/2} \tu}_2.
\end{align*}
Therefore 
\begin{equation}
I_1 \lesssim \norm{\rho^{1/2} \tu}_2 + 1.
\end{equation}
\item
Using Lemma \ref{T52} and the relation $\divg(\rho \tu) = \Delta G$, we obtain 
\begin{align*}
I_2 
&= - \frac{1}{\lambda+2\mu} \intO p\bu \cdot \nabla G \\
&\lesssim \norm{\rho^{1/C}\bu}_C 
\norm{\rho^{(C\gamma-1)/C}}_{(3C^2-6C)/(C^2-3C+6)} 
\norm{\nabla G}_{(3C-6)/(2C-6)} \\
&\lesssim \norm{\rho}_{(3C^2\gamma-6C\gamma-3C+6)/(C^2-3C+6)}^{(C\gamma-1)/C} 
\norm{\rho \tu}_{(3C-6)/(2C-6)}.
\end{align*}
An elementary calculation shows that for $8 \le C < \infty$
\[\frac{3C^2\gamma-6C\gamma-3C+6}{C^2-3C+6} \le \frac{2C\gamma-C+2}{2} \le \beta\] 
and hence $\norm{\rho}_{(3C^2\gamma-6C\gamma-3C+6)/(C^2-3C+6)} < \infty$.
Similarly, H\"older's inequality yields
\begin{equation}
I_2 \lesssim \norm{\rho \tu}_{(3C-6)/(2C-6)} \lesssim \norm{\rho^{1/2} \tu}_2.
\end{equation}
\item
H\"older's inequality yields
\begin{align*}
I_3
&= \intO \rho \tu \cdot (\bu \cdot \nabla) \bu \\
&\lesssim \norm{\rho^{1/2-1/C}}_{6C/(C-6)} \norm{\rho^{1/C}\bu}_C \norm{\rho^{1/2} \tu}_2 \norm{\nabla \bu}_3 \\
&\lesssim \norm{\rho}_{(3C-6)/(C-6)}^{(C-2)/(2C)} \norm{\rho^{1/2} \tu}_2 \norm{\nabla \bu}_3 \\
&\lesssim \norm{\rho^{1/2} \tu}_2 \norm{\nabla \bu}_3.
\end{align*}
Using Helmholz--Leray decomposition of $\bu$ and an interpolation, we obtain for each $\vep > 0$ there is a positive constant $\alpha_\vep$ such that 
\begin{align*}
\norm{\nabla \bu}_3
&\lesssim \norm{\divg \bu}_3 + \norm{\nabla\cP \bu}_3 \\
&\lesssim \vep \norm{\nabla^2 \cP \bu}_{5/3} + \alpha_\vep \\
&\lesssim \vep \norm{\rho \tu}_{5/3} + \alpha_\vep \\
&\le \vep \norm{\rho^{1/2}}_{10} \norm{\rho^{1/2}\tu}_2 + \alpha_\vep.
\end{align*}
An elementary calculation shows that for $8 \le C < \infty$
\[5 \le \frac{2C\gamma-C+2}{2} \le \beta\] 
and hence $\norm{\rho^{1/2}}_{10} = \norm{\rho}_{5}^{1/2} < \infty$.
Therefore, we have for each $\vep > 0$ there is a positive constant $\alpha_\vep$ such that  
\begin{equation}
I_3 \lesssim \norm{\rho^{1/2} \tu}_2 \norm{\nabla \bu}_3 
\lesssim \vep \norm{\rho^{1/2}\tu}_2^2 + \alpha_\vep
\end{equation}
\item 
Combining the identity \eqref{E42} and the estimates for $I_1$, $I_2$, and $I_3$ with a fixed small $\vep>0$, we obtain that for some positive constant $\alpha$
\begin{align*}
&\intO \rho |\tu|^2 + \frac{d}{dt} \intO \left(\mu |\nabla \bu|^2 +(\lambda+\mu) |\divg \bu|^2 - 2p \divg \bu + \frac{p^2}{\lambda+2\mu}\right) \\
&\le \frac{1}{4} \norm{\rho^{1/2}\tu}_2^2 + \alpha (\norm{\rho^{1/2}\tu}_2 + 1) \\
&\le \frac{1}{2} \norm{\rho^{1/2}\tu}_2^2 + \alpha^2 + \alpha. 
\end{align*}
Integrate in time over $[0,t]$ we obtain 
\begin{align*}
&\frac{1}{2} \int_0^t \intO \rho |\tu|^2 + \intO \left(\mu |\nabla \bu|^2 +(\lambda+\mu) |\divg \bu|^2 + \frac{p^2}{\lambda+2\mu}\right)(t) \\
&\le 2 \sup_{0 \le t \le T} \intO |p \divg \bu| + (\alpha^2 + \alpha)T + \alpha_0.
\end{align*}
where the constant $\alpha_0$ depends only on the initial data and the physical constants $\mu$ and $\lambda$.
Since $\intO |p \divg \bu| \le \norm{p}_{3/2} \norm{\divg \bu}_3 < \infty$, we conclude the result by taking supremum for $t \in [0,T]$.
\end{enumerate}
This completes the proof of Lemma \ref{T43}.
\end{proof}

We derive a uniform energy estimate for the material derivative by modifying the aguement in Sun--Wang--Zhang \cite{SWZ}.

\begin{lem}
\label{T44}
Let $8 \le C < \infty$.
If the condition \eqref{E40} holds for some $\beta$ satisfying 
\[\max\set{\frac{2C\gamma-C+2}{2}, \frac{3C-6}{C-6}} \le \beta,\]
then 
\[\sup_{0 \le t \le T} \intO \rho |\tu|^2 + \int_0^T \intO |\nabla \tu|^2 < \infty.\]
\end{lem}

\begin{proof}
We divide the proof into a few steps.
\begin{enumerate}
[\bf{Step} 1)]
\item
Taking the material derivative to $L\bu$ yields 
\[L(\partial_t \bu) + \divg (L\bu \otimes \bu) - \divg \bu L\bu.\]
Taking the material derivative to $\rho \tu  + \nabla p$ yields 
\begin{align*}
&\partial_t (\rho \tu ) + (\bu \cdot \nabla) (\rho \tu ) 
+ \partial_t \nabla p + (\bu \cdot \nabla) \nabla p \\
&= - \divg (\rho \bu) \tu  + \rho \partial_t \tu  
+ (\nabla \rho \cdot \bu) \tu  + \rho (\bu \cdot \nabla) \tu  
+ \nabla \partial_t p + \divg (\nabla p \otimes \bu) - (\nabla p) \divg \bu \\
&= \rho \partial_t \tu  + \rho (\bu \cdot \nabla) \tu  
+ \nabla \partial_t p + \divg (\nabla p \otimes \bu) - \divg \bu L\bu.
\end{align*}
Hence, from the momentum equation, we obtain  
\[[\rho \partial_t \tu  + \rho (\bu \cdot \nabla) \tu ] 
- [L(\partial_t \bu) + \divg (L\bu \otimes \bu)]
= - [\nabla \partial_t p + \divg (\nabla p \otimes \bu)].\]
Multiplying $\tu $ and integrating on $\Omega$ yields 
\begin{equation}
\label{E43}
\frac{1}{2} \frac{d}{dt} \intO \rho |\tu |^2 
- \intO [L(\partial_t \bu) + \divg (L\bu \otimes \bu)] \cdot \tu 
= \intO \partial_t p \divg \tu  + (\bu \cdot \nabla) \tu  \cdot \nabla p.
\end{equation}
\item
Now, we estimate the nonlinear terms.
Integrating by parts we have 
\begin{align*}
&\intO |\nabla \tu|^2 + \intO [\Delta \partial_t \bu + \divg (\Delta u \otimes \bu)] \cdot \tu \\
&= \intO [-\Delta ((\bu \cdot \nabla) \bu) + \divg (\Delta u \otimes \bu)] \cdot \tu \\
&\ll \intO |\nabla \bu|^2 |\nabla \tu|.
\end{align*}
Similarly, we have 
\begin{align*}
&\intO |\divg \tu|^2 + \intO [\nabla \divg \partial_t \bu + \divg (\nabla \divg \bu \otimes \bu)] \cdot \tu \\
&= \intO [-\nabla \divg((\bu \cdot \nabla) \bu) + \divg (\nabla \divg \bu \otimes \bu)] \cdot \tu \\
&\ll \intO |\nabla \bu|^2 |\nabla \tu|.
\end{align*}
Therefore, the Cauchy inequality yields for $\vep > 0$ 
\begin{equation}
\label{E44}
\begin{split}
&\mu \intO |\nabla \tu|^2 
+ (\lambda+\mu) \intO |\divg \tu|^2  
+ \intO [L(\partial_t \bu) + \divg (L\bu \otimes \bu)] \cdot \tu \\
&\ll \intO |\nabla \bu|^2 |\nabla \tu| 
\ll \vep \intO |\nabla \tu|^2 + \alpha_\vep \intO |\nabla \bu|^4
\end{split}
\end{equation}
for some positive number $\alpha_\vep$.
For the computations in Step 1 and 2, we follow the lines in Sun--Wang--Zhang \cite{SWZ}.
\item
For the barotropic gases, we can easily verify the realtion 
\[\p_t p + \divg (p\bu) = - (\gamma-1) p \divg \bu\]
by a direct computation.
Hence, integrating by parts yields 
\begin{align*}
&\intO \partial_t p \divg \tu + (\bu \cdot \nabla) \tu  \cdot \nabla p \\
&= \intO \partial_t p \divg \tu - p \bu \cdot \nabla \divg \tu - p (\nabla \bu)^T : \nabla \tu \\
&= \intO (\partial_t p + \divg (p \bu)) \divg \tu - p (\nabla \bu)^T : \nabla \tu \\
&\le (\norm{\partial_t p + \divg (p \bu)}_2 + \norm{p \nabla \bu}_2) \norm{\nabla \tu}_2 \\
&\ll \norm{p}_4 \norm{\nabla \bu}_4 \norm{\nabla \tu}_2.
\end{align*}
An elementary calculation shows that for $8 \le C < \infty$
\[4\gamma < \frac{2C\gamma-C+2}{2}\]
and hence $\norm{p}_4 \le \norm{\rho}_{4\gamma}^\gamma < \infty$.
By the Cauchy inequality we obtain that for $\vep > 0$
\begin{equation}
\label{E45}
\intO \partial_t p \divg \tu + (\bu \cdot \nabla) \tu  \cdot \nabla p 
\ll \vep \intO |\nabla \tu|^2 + \alpha_\vep \left(\intO |\nabla \bu|^4\right)^{1/2}
\end{equation}
for some positive number $\alpha_\vep$.
\item
Combining \eqref{E43}, \eqref{E44}, and \eqref{E45} with a fixed small number $\vep$, we get
\[\frac{d}{dt} \intO \rho |\tu|^2  + \intO |\nabla \tu|^2 
\ll \intO |\nabla \bu|^4 + 1.\]
Using the Helmholz--Leray decomposition of $\bu$ and an interpolation we have 
\begin{align*}
\norm{\nabla \bu}_4 
&\lesssim \norm{\divg \bu}_4 + \norm{\nabla\cP u}_4 
\lesssim \norm{G-\inn{G}}_4 + \norm{p-\inn{p}}_4 + \norm{\nabla\cP u}_4 \\
&\lesssim \norm{\nabla G}_{12/7} + \norm{\nabla^2 \cP u}_{12/7} + \norm{p}_4 
\lesssim \norm{\rho \tu}_{12/7} + 1 \\
&\lesssim \norm{\rho^{1/2}}_{12} \norm{\rho^{1/2}\tu}_2 + 1 
\lesssim \norm{\rho^{1/2}\tu}_2 + 1.
\end{align*}
If we denote 
\begin{align*}
Y(t) &= \intO \rho |\tu|^2(t) \\
Z(t) &= \intO |\nabla \tu|^2(t),
\end{align*}
then we have, for some positive constant $\alpha$, the inequality 
\begin{equation}
\label{E46}
\frac{d}{dt} Y(t) + Z(t)  \le \alpha Y(t)^2 + \alpha.
\end{equation}
\item
We observe that 
\begin{align*}
&\frac{d}{dt} \left(Y(t) \exp\left(-\alpha \int_0^t Y(s) ds\right)\right) \\
&= \left(\frac{d}{dt} Y(t) - \alpha Y(t)^2\right) 
\exp\left(-\alpha \int_0^t Y(s) ds\right) \\
&\le \alpha \exp\left(-\alpha \int_0^t Y(s) ds\right).
\end{align*}
Integrating over $[0,\tau]$ yields 
\[Y(\tau) \le Y(0) \exp\left(\alpha \int_0^\tau Y(s) ds\right) + \alpha \int_0^\tau \exp\left(\alpha \int_t^\tau Y(s) ds\right) dt.\]
From Lemma \ref{T43} 
\[\eta := \int_0^T Y(t) dt = \int_0^T \intO \rho |\tu|^2 < \infty.\]
Thus, for all $\tau \in [0,T]$
\[Y(\tau) \le Y(0) e^{\alpha \eta} + \alpha e^{\alpha \eta} T.\]
Since $\tau \in [0,T]$ is arbitrary, we obtain that 
\[\sup_{0 \le \tau \le T} Y(\tau) < \infty.\]
Moreover, we can use the inequality \eqref{E46} again to conclude that 
\[\int_0^T Z(t) dt = \int_0^T \intO |\nabla \tu|^2 \le \alpha T\left(1 + \sup_{0 \le t \le T} Y(t)^2\right).\]
\end{enumerate}
This completes the proof.
\end{proof}

The following lemma show that the higher integrability of $\rho$ and $\rho^{1/2} \tu$ yields the higher integrability of $\nabla \bu$. 

\begin{lem}
\label{T45}
Let $3< A < 6$.
Suppose 
\[\sup_{0 \le t \le T} \norm{\rho^{1/2}\tu}_2 < \infty\]
and $\sup_{0 \le t \le T} \norm{\rho}_\beta < \infty$ for some $\beta$ satisfying 
\[\max\set{A\gamma, \frac{3A}{6-A}} \le \beta.\]
Then 
\[\sup_{0 \le t \le T} \norm{\nabla \bu}_A < \infty.\]
\end{lem}

\begin{proof}
We recall first that $(\lambda+2\mu) \divg \bu = p+G$ and $\lambda+2\mu>0$.
Because of the periodic boundary condition, we have $\intO \divg \bu  = 0$ and hence 
\[\inn{G} + \inn{p} = 0\]
where we use the notation $\inn{f} = \fint_\Omega f$.
We next recall the relations $\divg(\rho \tu) = \Delta G$ and $\cP (\rho \tu ) = \mu \Delta \cP \bu$. 
Using the Helmholz--Leray decomposition of $\bu$ and the elliptic estimate for $G$ and $\cP \bu$ we obtain that 
\begin{align*}
\norm{\nabla \bu}_A 
&\lesssim \norm{\divg \bu}_A + \norm{\nabla\cP \bu}_A \\
&\lesssim \norm{G-\inn{G}}_A + \norm{p-\inn{p}}_A + \norm{\nabla\cP \bu}_A \\
&\lesssim \norm{\nabla G}_{3A/(3+A)} + \norm{\nabla^2 \cP \bu}_{3A/(3+A)} + \norm{p}_A \\ 
&\lesssim \norm{\rho \tu}_{3A/(3+A)} + \norm{p}_A.
\end{align*}
Since we have $\norm{p}_A = \norm{\rho}_{A\gamma}^\gamma$ and 
\[\norm{\rho \tu}_{3A/(3+A)} \le \norm{\rho^{1/2}}_{6A/(6-A)} \norm{\rho^{1/2}\tu}_2 \le \norm{\rho}_{3A/(6-A)}^{1/2} \norm{\rho^{1/2}\tu}_2\]
by H\"older inequality, we conclude that 
\[\norm{\nabla \bu}_A \lesssim \norm{\rho^{1/2}\tu}_2 + 1.\]
Integrating in time yields the result.
\end{proof}

Now, we are ready to prove Theorem \ref{G1}.

\begin{proof}[\textbf{Proof of Theorem \ref{G1}}] 
Let $3< A < 6 < B \le 8 \le C < \infty$ with $1/A + 1/B < 1/3$.
Suppose the condition \eqref{E40} holds for 
\[\max\set{A\gamma, \frac{3A}{6-A}, \frac{2C\gamma-C+2}{2}, \frac{3C-6}{C-6}, \frac{(C-1)B}{C-B}} \le \beta.\] 
Combining Lemma \ref{T42}, Lemma \ref{T43}, Lemma \ref{T44}, and Lemma \ref{T45} we obtain that    
\[\sup_{0 \le t \le T} \left(\norm{\nabla \bu}_A + \norm{\rho}_\beta + \norm{\rho^{1/C}\bu}_C\right) < \infty.\]
Therefore by Lemma \ref{T31} $\sup_{0 \le t \le T} \norm{\rho}_\infty < \infty$.
\end{proof}

\begin{rem}
If $\gamma=3/2$, then we may take $A=5$, $B=8$, $C=15$, and $\beta = 16$ so that all the conditions of lemmas are fulfilled.
\end{rem}
		
\section{Proof of Theorem \ref{G2}}
\label{S5}

In this section, we assume that $(\rho, \bu,\theta)$ is the strong solution with the regularity stated in Theorem \ref{T2} and satisfies for some positive number $\delta$ the following integrability condition 
\begin{equation}
\label{E50}
\sup_{0 \le t \le T} \left(\norm{\rho}_{\delta} + \norm{\Delta \theta}_{2} + \norm{\divg \bu}_3\right) < \infty.
\end{equation}
The constant $\delta$ will be specified in each lemmas.

\begin{lem}
\label{T51}
If the condition \eqref{E50} holds for some $\delta \ge 3/2$, then
\[\sup_{0 \le t \le T} \intO \rho(|\bu|^2 + |\theta|^2) 
+ \int_0^T \intO (|\nabla \bu|^2 +|\nabla \theta|^2) < \infty.\]
\end{lem}

\begin{proof}
Testing with $\bu$ in the equation \eqref{E12} yields
\[\frac{1}{2} \frac{d}{dt} \intO \rho |\bu|^2 + \mu \intO |\nabla \bu|^2 + (\mu+\lambda) \intO |\divg \bu|^2 = \intO p \divg \bu.\]
By H\"older's inequality 
\[\intO p \divg \bu
\le \norm{\rho}_{3/2} \norm{\theta}_\infty \norm{\divg \bu}_3 < \infty.\]
Similarly, testing with $\theta$ in the equation \eqref{E16} yields
\begin{align*}
&\frac{1}{2} \frac{d}{dt} \intO \rho |\theta|^2 +  \intO |\nabla \theta|^2 \\
&= - \intO \rho \theta^2 \divg \bu 
+ \frac{\mu}{2} \intO |\nabla \bu +\nabla \bu^{t}|^2 \theta 
+ \lambda \intO |\divg \bu|^2 \theta.
\end{align*}
The right hand side is bounded by 
\[\norm{\rho}_{3/2} \norm{\theta}_\infty^2 \norm{\divg \bu}_3 + \norm{\theta}_\infty \norm{\nabla \bu}_2^2.\]
Integrating in time yields the result.
\end{proof}

\begin{lem}
\label{T52}
Let $2 < C < \infty$.
If the condition \eqref{E50} holds for some $\delta$ satisfying $(C+2)/2 \le \delta$, then 
\[\sup_{0 \le t \le T} \norm{\rho^{1/C}\bu}_C < \infty.\]
\end{lem}

\begin{proof}
From Lemma \ref{T41} 
\[\frac{d}{dt} \intO \rho |\bu|^C 
\lesssim 1 + \intO p^2 |\bu|^{C-2}.\]
By H\"older's inequality 
\[\intO p^2 |\bu|^{C-2} 
= \intO \rho^2 \theta^2 |\bu|^{C-2} 
\le \norm{\theta}_\infty^2 \left(\intO \rho^{(C+2)/2}\right)^{2/C} 
\left(\intO \rho |\bu|^C\right)^{(C-2)/C}.\]
By Young's inequality 
\[\frac{d}{dt} \intO \rho |\bu|^C \lesssim 1 + \intO \rho |\bu|^C.\]
The result follows by the Gronwall lemma.
\end{proof}

\begin{lem}
\label{T53}
Let $10 \le C < \infty$.
If the condition \eqref{E50} holds for some $\delta$ satisfying $(C+2)/2 \le \delta$, then 
\[\int_0^T \intO \rho |\tu|^2 + \sup_{0 \le t \le T} \intO |\nabla \bu|^2 < \infty.\]
\end{lem}

\begin{proof}
We divide the proof into a few steps.
\begin{enumerate}[\bf{Step} 1)]
\item
Recall the following identity, which was derived in the first step proving Lemma \ref{T43},
\begin{equation}
\label{E51}
\begin{split}
&\intO \rho |\tu|^2 + \frac{1}{2} \frac{d}{dt} \intO \left(\mu |\nabla \bu|^2 +(\lambda+\mu) |\divg \bu|^2 - 2p\divg \bu + \frac{1}{\lambda+2\mu} p^2\right) \\
&= - \frac{1}{\lambda+2\mu} \intO (\p_t p + \divg (p\bu)) G 
- \frac{1}{\lambda+2\mu} \intO p\bu \cdot \nabla G 
+ \intO \rho \tu \cdot (\bu \cdot \nabla) \bu \\
&=: I_1 + I_2 + I_3.
\end{split}
\end{equation}
\item
The equation \eqref{E16} can be rewritten as  
\[\p_t p + \divg(p\bu) = - p \divg \bu +  \Delta \theta + \frac{\mu}{2} |\nabla \bu +\nabla \bu^{t}|^2 + \lambda |\divg \bu|^2.\]
Since $\norm{p}_3 \le \norm{\rho}_3 \norm{\theta}_\infty < \infty$, the condition \eqref{E50} yields
\begin{align*}
I_1 
&\lesssim \intO (|p \divg \bu| +  |\Delta \theta| + |\nabla \bu|^2) |G| \\
&\lesssim \left(\norm{p \divg \bu}_{3/2} + \norm{\Delta \theta}_{3/2} + \norm{|\nabla \bu|^2}_{3/2}\right) \norm{G}_3 \\
&\lesssim \left(\norm{p}_3 \norm{\divg \bu}_3 + \norm{\Delta \theta}_2 + \norm{\nabla \bu}_3^2\right) \left(\norm{\divg \bu}_3+\norm{p}_3\right) \\
&\lesssim 1 + \norm{\nabla \bu}_3^2. 
\end{align*}
Using Lemma \ref{T52} and the relation $\divg(\rho \tu) = \Delta G$, we obtain 
\begin{align*}
I_2 
&\lesssim \intO |p\bu||\nabla G| \\
&\le \norm{\theta}_\infty \norm{\rho^{1/C}\bu}_C \norm{\rho^{(C-1)/C}}_{(3C^2-2C)/(C^2-3C+2)} \norm{\nabla G}_{(3C-2)/(2C-2)} \\
&\lesssim \norm{\rho}_{(3C-2)/(C-2)}^{(C-1)/C} \norm{\rho \tu}_{(3C-2)/(2C-2)}.
\end{align*}
An elementary calculation shows that for $10 \le C < \infty$
\[(3C-2)/(C-2) \le (C+2)/2 \le \delta\] 
and hence by H\"older's inequality 
\[I_2 \lesssim \norm{\rho \tu}_{(3C-2)/(2C-2)} 
\lesssim \norm{\rho^{1/2}}_{(6C-4)/(C-2)} \norm{\rho^{1/2} \tu}_2 
\lesssim \norm{\rho^{1/2} \tu}_2.\]
An elementary calculation shows also that for $10 \le C < \infty$
\[(3C-6)/(C-6) \le (C+2)/2 \le \delta\] 
and hence by H\"older's inequality 
\begin{align*}
I_3
&\lesssim \norm{\rho^{1/2-1/C}}_{6C/(C-6)} \norm{\rho^{1/C}\bu}_C \norm{\rho^{1/2} \tu}_2 \norm{\nabla \bu}_3 \\
&\lesssim \norm{\rho}_{(3C-6)/(C-6)}^{(C-2)/(2C)} \norm{\rho^{1/2} \tu}_2 \norm{\nabla \bu}_3 \\
&\lesssim \norm{\rho^{1/2} \tu}_2 \norm{\nabla \bu}_3.
\end{align*}
\item
Combining \eqref{E51} and the estimates for $I_1$, $I_2$, and $I_3$, we obtain that 
\begin{equation}
\label{E52}
\begin{split}
&\intO \rho |\tu|^2 + \frac{1}{2} \frac{d}{dt} \intO \left(\mu |\nabla \bu|^2 +(\lambda+\mu) |\divg \bu|^2 - 2p\divg \bu + \frac{1}{\lambda+2\mu} p^2\right) \\
&\lesssim 1 + \norm{\nabla \bu}_{3}^2 
+ \norm{\rho^{1/2} \tu}_2 + \norm{\rho^{1/2} \tu}_2 \norm{\nabla \bu}_3.
\end{split}
\end{equation}
Using the Helmholz--Leray decomposition of $\bu$ and an interpolation we have 
\[\norm{\nabla \bu}_3
\lesssim \norm{\divg \bu}_3 + \norm{\nabla\cP \bu}_3 
\lesssim \vep \norm{\nabla^2 \cP \bu}_{8/5} + \alpha_\vep\]
for some positive constant $\alpha_\vep$.
Using the relation $\cP (\rho \tu ) = \mu \Delta \cP \bu$ in \eqref{E21} we have 
\[\norm{\nabla^2 \cP \bu}_{8/5} 
\lesssim \norm{\rho \tu}_{8/5} 
\le \norm{\rho^{1/2}}_8 \norm{\rho^{1/2}\tu}_2
\lesssim \norm{\rho^{1/2}\tu}_2.\]
Hence 
\[\norm{\nabla \bu}_3 \lesssim \vep \norm{\rho^{1/2}\tu}_2 + \alpha_\vep.\]
We fix a small number $\vep$ so that, by Young's inequality, the estimate \eqref{E52} becomes 
\begin{align*}
&\intO \rho |\tu|^2 + \frac{d}{dt} \intO \left(\mu |\nabla \bu|^2 +(\lambda+\mu) |\divg \bu|^2 - 2p\divg \bu + \frac{1}{\lambda+2\mu} p^2\right) \\
&\le \frac{1}{2} \intO \rho |\tu|^2 + \alpha
\end{align*}
for some constant $\alpha$.
Since $\intO |p \divg \bu| \le \norm{p}_{3/2} \norm{\divg \bu}_3 < \infty$, we conclude the result by integrating in time over $[0,t]$ and then by taking supremum for $t \in [0,T]$.
\end{enumerate}
This completes the proof of Lemma \ref{T53}.
\end{proof}

\begin{lem}
\label{T54}
Let $10 \le C < \infty$.
If the condition \eqref{E50} holds for some $\delta$ satisfying $(C+2)/2 \le \delta$, then 
\[\sup_{0 \le t \le T} \intO \rho |\tu|^2 + \int_0^T \intO |\nabla \tu|^2 < \infty.\]
\end{lem}

\begin{proof}
Since the proof is essentially the same as the proof of Lemma \ref{T44}, we only describe the different parts.
Taking the material derivative to the momentum equation we obtain 
\begin{equation}
\label{E53}
\frac{1}{2} \frac{d}{dt} \intO \rho |\tu |^2 
- \intO [L(\partial_t \bu) + \divg (L\bu \otimes \bu)] \cdot \tu 
= \intO \partial_t p \divg \tu  + (\bu \cdot \nabla) \tu  \cdot \nabla p.
\end{equation}
Integrating by parts yields
\begin{equation}
\label{E54}
\begin{split}
&\mu \intO |\nabla \tu|^2 
+ (\lambda+\mu) \intO |\divg \tu|^2  
+ \intO [L(\partial_t \bu) + \divg (L\bu \otimes \bu)] \cdot \tu \\
&\ll \intO |\nabla \bu|^2 |\nabla \tu|
\end{split}
\end{equation}
and  
\begin{align*}
&\intO \partial_t p \divg \tu + (\bu \cdot \nabla) \tu  \cdot \nabla p \\
&= \intO \partial_t p \divg \tu - p \bu \cdot \nabla \divg \tu - p (\nabla \bu)^T : \nabla \tu \\
&= \intO (\partial_t p + \divg (p \bu)) \divg \tu - p (\nabla \bu)^T : \nabla \tu \\
&\le (\norm{\partial_t p + \divg (p \bu)}_2 + \norm{p \nabla \bu}_2) \norm{\nabla \tu}_2
\end{align*}
The equation \eqref{E16} can be rewritten as  
\[\partial_t p + \divg (p \bu) 
= - p \divg \bu +  \Delta \theta + \frac{\mu}{2} |\nabla \bu +\nabla \bu^{t}|^2 + \lambda |\divg \bu|^2.\]
Since $\norm{p}_4 \le \norm{\rho}_4 \norm{\theta}_\infty < \infty$, we have 
\begin{align*}
&\norm{\partial_t p + \divg (p \bu)}_2 + \norm{p \nabla \bu}_2 \\
&\ll \norm{p}_4 \norm{\nabla \bu}_4 + \norm{\Delta \theta}_2 + \norm{|\nabla \bu|^2}_2 \\
&\ll 1 + \norm{\nabla \bu}_4^2.
\end{align*}
By the Cauchy inequality we obtain that for $\vep > 0$
\begin{equation}
\label{E55}
\intO \partial_t p \divg \tu + (\bu \cdot \nabla) \tu  \cdot \nabla p 
\ll \vep \intO |\nabla \tu|^2 + \alpha_\vep \left(1+\intO |\nabla \bu|^4\right)
\end{equation}
for some positive number $\alpha_\vep$.
Combining \eqref{E53}, \eqref{E54}, and \eqref{E55} with a fixed small number $\vep$, we get
\[\frac{d}{dt} \intO \rho |\tu|^2  + \intO |\nabla \tu|^2 
\ll \intO |\nabla \bu|^4 + 1.\]
Now, the result follows by the Gronwall lemma.
\end{proof}

\begin{lem}
\label{T55}
Let $3< A < 6$.
Suppose 
\[\sup_{0 \le t \le T} \norm{\rho^{1/2}\tu}_2 < \infty\]
and $\sup_{0 \le t \le T} \norm{\rho}_\beta < \infty$ for some $\beta$ satisfying 
\[\max\set{A, \frac{3A}{6-A}} \le \beta.\]
Then 
\[\sup_{0 \le t \le T} \norm{\nabla \bu}_A < \infty.\]
\end{lem}

\begin{proof}
Using the Helmholz--Leray decomposition of $\bu$ and the Sobolev-type estimate we obtain 
\begin{align*}
\norm{\nabla \bu}_A 
&\lesssim \norm{\divg \bu}_A + \norm{\nabla\cP \bu}_A \\
&\lesssim \norm{G-\inn{G}}_A + \norm{p-\inn{p}}_A + \norm{\nabla\cP \bu}_A \\
&\lesssim \norm{\nabla G}_{3A/(3+A)} + \norm{\nabla^2 \cP \bu}_{3A/(3+A)} + \norm{p}_A \\ 
&\lesssim \norm{\rho \tu}_{3A/(3+A)} + \norm{p}_A.
\end{align*}
We have $\norm{p}_A \le \norm{\rho}_A \norm{\theta}_\infty$ and 
by H\"older inequality 
\[\norm{\rho \tu}_{3A/(3+A)} \le \norm{\rho^{1/2}}_{6A/(6-A)} \norm{\rho^{1/2}\tu}_2 \le \norm{\rho}_{3A/(6-A)}^{1/2} \norm{\rho^{1/2}\tu}_2\]
and hence 
\[\norm{\nabla \bu}_A \lesssim \norm{\rho^{1/2}\tu}_2 + 1.\]
Integrating in time yields the result.
\end{proof}

Now, we are ready to prove Theorem \ref{G2}.

\begin{proof}[\textbf{Proof of Theorem \ref{G2}}] 
Let $3< A < 6 < B < C < \infty$ with $1/A + 1/B < 1/3$.
Suppose the condition \eqref{E50} holds for 
\[\max\set{\frac{3A}{6-A}, \frac{C+2}{2}, \frac{(C-1)B}{C-B}} \le \beta.\] 
Combining Lemma \ref{T52}, Lemma \ref{T53}, Lemma \ref{T54}, and Lemma \ref{T55} we obtain that    
\[\sup_{0 \le t \le T} \left(\norm{\nabla \bu}_A + \norm{\rho}_\beta + \norm{\rho^{1/C}\bu}_C\right) < \infty.\]
Therefore by Lemma \ref{T31} $\sup_{0 \le t \le T} \norm{\rho}_\infty < \infty$.
\end{proof}

\begin{rem}
We may take $A=24/5$, $B=41/5$, $C=22$, and $\beta = 13$ so that all the conditions of lemmas are fulfilled.
\end{rem}

\section*{Acknowledgment}

Hi Jun Choe has been supported by the National Research Foundation of Korea (NRF) grant funded by the Korea government(MSIP) (No. 2015R1A5A1009350).
Minsuk Yang has been supported by the National Research Foundation of Korea (NRF) grant funded by the Korea government(MSIP) (No. 2016R1C1B2015731).

\end{document}